\documentclass[reqno, english]{amsart}
\usepackage{etex}
\usepackage{amsmath,amssymb,amsthm,bbm,mathtools,comment}
\usepackage[shortlabels]{enumitem}
\usepackage[pdftex,colorlinks,backref=page,citecolor=blue]{hyperref}
\hypersetup{pdfpagemode=UseNone,pdfstartview={XYZ null null 1.00}}
\usepackage[mathscr]{euscript}
\usepackage[usenames,dvipsnames]{color}
\usepackage{adjustbox,tikz,calc,graphics,babel,standalone}
\usetikzlibrary{shapes.misc,calc,intersections,patterns,decorations.pathreplacing}
\usetikzlibrary{arrows,shapes,positioning}
\usetikzlibrary{decorations.markings}
\usepackage[final]{microtype}
\usepackage[numbers]{natbib}
\usepackage{cmtiup}
\usepackage{amsfonts}
\usepackage{graphicx}
\usepackage{caption}
\usepackage{subcaption}
\usepackage{verbatim}
\usepackage{array}
\usepackage[frame,cmtip,arrow,matrix,line,graph,curve]{xy}
\usepackage{graphpap, color, pstricks}
\usepackage{pifont}
\usepackage[final]{microtype}
\usepackage{cmtiup}
\usepackage{todonotes}
\usetikzlibrary{topaths,calc} 
\tikzstyle{vertex} = [fill,shape=circle,node distance=80pt]
\tikzstyle{edge} = [opacity=0.4,fill opacity=0.0,line cap=round, line join=round, line width=40pt]
\tikzstyle{elabel} =  [fill,shape=circle,node distance=30pt]

\allowdisplaybreaks

\theoremstyle{plain}
\newtheorem{theorem}{Theorem}[section]		
\newtheorem{lemma}[theorem]{Lemma}
\newtheorem*{claim*}{Claim}
\newtheorem{claim}[theorem]{Claim}

\newtheorem{observation}[theorem]{Observation}
\newtheorem{corollary}[theorem]{Corollary}

\newtheorem{problem}[theorem]{Problem}

\theoremstyle{remark}

\newcommand{\eps}{\ensuremath{\varepsilon}}
\let\emptyset\varnothing

\let\originalleft\left
\let\originalright\right
\renewcommand{\left}{\mathopen{}\mathclose\bgroup\originalleft}
\renewcommand{\right}{\aftergroup\egroup\originalright}

\makeatletter
\def\imod#1{\allowbreak\mkern10mu({\operator@font mod}\,\,#1)}
\makeatother

\title{A note on unavoidable patterns in locally dense colourings}
\author{Ant\'onio Gir\~ao}
\author{David Munh\'a Correia}

\thanks{
AG: Mathematical Institute, University of Oxford, Oxford OX2 6GG, UK. E-mail: {\tt girao@maths.ox.ac.uk}. { EPSRC grant EP/V007327/1}}
\thanks{
DMC: Department of Mathematics, ETH, Z\"urich, Switzerland. Research supported in part by SNSF grant 200021\_196965. E-mail: {\tt david.munhacanascorreia@math.ethz.ch}}

\begin{document}
\maketitle
\begin{abstract}
We show that there is a constant $C$ such that for every $\varepsilon>0$ any $2$-coloured $K_n$ with minimum degree at least $n/4+\varepsilon n$ in both colours contains a  complete subgraph on $2t$ vertices where one colour class forms a $K_{t,t}$, provided that $n\geq \varepsilon^{-Ct}$. Also, we prove that if $K_n$ is $2$-coloured with minimum degree at least $\varepsilon n$ in both colours then it must contain one of two natural colourings of a complete graph. Both results are tight up to the value of $C$ and they answer two recent questions posed by Kam\v{c}ev and M\"{u}yesser. 
\end{abstract}
\section{Introduction}

Ramsey theory is concerned with finding which structures are unavoidable in any finite partition of a complex enough structure. An example of this is the well known Ramsey Theorem which states that in any $r$-colouring of the edges of a $K_n$ there is a monochromatic $K_t$, provided $n$ is sufficiently large compared with $r$ and $t$. The problem of estimating the smallest such $n$ is perhaps one the most important open problems in the area. 

More recently, there has been a lot of interest in understanding what other coloured patterns of a clique must appear in any $2$-colouring of a $K_n$ if we instead impose extra conditions on the colouring. Cutler and Mont\'agh~\cite{CM}, answering a question of Bollob\'as, proved that for every $0<\varepsilon<1/2$ and $t\in \mathbb{N}$, there is $n(\varepsilon,t)$ such that in any $2$-edge colouring of $K_n$ (with $n\geq n(\varepsilon, t)$) where each colour appears in at least $\varepsilon n^2$ edges there must exist $2t$ vertices in which one of the colours induces a $K_t$ or two vertex disjoint $K_t$'s. 
Later, Fox and Sudakov~\cite{FS} gave essentially tight bounds showing that $n(\varepsilon, t)=\varepsilon^{-O(t)}$. 
Since then, there have been several extensions and variants of this result e.g., for multiple colours (\cite{colorfulpaths}, \cite{BLM}), replacing $K_n$ by a dense host graph (\cite{muyesser2020turan}), with $\varepsilon$ replaced by a function of $n$ (\cite{BLM}, \cite{caro2021unavoidable},\cite{girao2022two}) and for tournaments~\cite{girao2022turan}. 

We investigate yet another variant of the above problem raised by Wesley Pegden and studied recently by Kam\v{c}ev and M\"{u}yesser in \cite{KM}. They considered the following problem: which coloured patterns are unavoidable in $2$-colourings of $K_n$ such that each colour induces a graph with large minimum degree? They found that there exists a phase transition concerning what 'large' means. More precisely, they first showed that whenever $K_n$ is $2$-coloured with the property that every vertex is incident with at least $n/4+\varepsilon n$ edges in each colour there must exist a $K_{2t}$ where one colour induces a $K_{t,t}$, provided $n$ is sufficiently large compared with $\varepsilon, t$. 

Moreover, note that this behaviour no longer persists if the minimum degree of each colour is less than $n/4$. Indeed, let us say that a $2$-colouring of $K_{4n}$ is in $\mathcal{P}_n$ if there is a partition of the vertex set into $4$ sets, $V_1\cup V_2\cup V_3\cup V_4$, each of size $n$ such that the following holds: all edges with both endpoints in $V_1\cup V_2$ are red and all edges with both endpoints in $V_3\cup V_4$ blue; further, all edges with one endpoint in $V_1$ and another in $V_3$ are red and similarly all edges with one endpoint in $V_2$ and the other in $V_4$; the remaining edges are coloured blue. It is easy to see that any $2$-coloured complete graph in $\mathcal{P}_n$ for any $n$ does not contain a $K_{2t}$ where one colour induces a $K_{t,t}$ for any $t\in \mathbb{N}$. Therefore, for a lesser restriction on the minimum degree the class of unavoidable patterns should be larger. Indeed, Kam\v{c}ev and M\"{u}yesser in \cite{KM} showed the following. Let a $2$-coloured complete graph $G$ be a \emph{$t$-locally-unavoidable pattern} if one of the following holds.
\begin{itemize}
    \item The edges in one of the colours induce a $K_{t,t}$.
    \item $G \in P_t$.
\end{itemize}
\noindent Then, any $2$-colouring of $K_n$ such that each colour induces a graph of minimum degree at least $\eps n$ contains a $t$-locally-unavoidable pattern, provided that $n$ is sufficiently large compared to $\eps,t$.

\noindent In this note, we give tight bounds for both of these results. Firstly, we show the following.
\begin{theorem}\label{thm:1}
Let $1/4 > \eps > 0$, $C$ be an arbitrarily large constant and $t,n$ be such that $n > \eps^{-Ct}$. Then, every $2$-coloured $K_n$ such that every vertex has degree at least $n/4+\eps n$ in each colour contains an induced copy of $K_{t,t}$ in some colour.
\end{theorem}
\noindent We remark that this is tight up to the constant in the exponent. To see this, take a colouring of $K_{4n}$ in $P_{n}$ where $n\coloneqq \varepsilon^{-t/4}$. Now, we randomly recolour each edge inside $V_1\cup V_2$ to blue with probability $2\varepsilon$ independently, we do the same for all the edges in $V_3\cup V_4$. A simple concentration argument shows that w.h.p. the minimum degree in both colours is at least $n/4+\varepsilon n $ and there are no blue copies of $K_t, K_{t,t}\subset V_1\cup V_2$  or red copies $K_t,K_{t,t}\subset V_3 \cup V_4$. Secondly, we show the following.
\begin{theorem}\label{thm:2}
Let $1/2 > \eps > 0$, $C$ be an arbitrarily large constant and $t,n$ be such that $n > \eps^{-Ct}$. Then, every $2$-coloured $K_n$ such that every vertex has degree at least $\eps n$ in each colour contains a $t$-locally-unavoidable pattern.
\end{theorem}
\noindent Like before, it is easy to see that the above is tight up to the constant in the exponent. 
\section{Preliminaries}
\subsection{Notation}
Given a red/blue edge colouring of $K_n$ and a vertex $v$, we shall write $d_r(v)$ to denote the size of the neighbourhood of $x$ in red and similarly $d_b(v)$ the size of the neighbourhood of $v$ in blue.  We say $H\subset K_n$ forms an alternating homogeneous $t$-blowup of an alternating $C_4$ if $V(H)=A_0\cup A_1\cup A_2 \cup A_3$, $|A_i|=t$ and $A_i$ induces a monochromatic $K_t$, for all $i\in \{0,\ldots, 3\}$. Moreover, $E[A_i, A_{i+1}]$ is  monochromatic in red if $i\in \{0,2\}$ and monochromatic in blue if $i\in \{1,3\}$.  
We will often for simplicity of notation omit floors and ceilings.
\subsection{Standard tools}
\begin{lemma}[Chernoff bound]\label{lem:chernoff}
Let $X_1, \ldots X_n$ be independent random variables where $X_i\in \{0,1\}$. Let $X\coloneqq \sum_{i=1}^{n} X_i$. Then, 
$\mathbb{P}[|X-\mathbb{E}[X]|\geq t]\leq 2e^{-t^2/4} .$
\end{lemma}
\begin{lemma}[Ramsey's Theorem]\label{lem:ramsey}
Let $K_n$ be $2$-coloured and $n \geq 2^{2t}$. Then, it contains a monochromatic clique of size $t$.
\end{lemma}
\begin{lemma}\label{prop:KST}
Let $\alpha \in (0,1)$ and $a,k, b \in \mathbb{N}$ be such that $a \geq \frac{k}{\alpha}$. Let $G$ be a bipartite graph with parts $A,B$ such that $|A| = a$ and $|B| = b$ and suppose that $e(A,B) \geq \alpha ab$. Then, there exists a subset $S \subseteq A$ of size $k$ with at least $\left( \frac{\alpha}{e} \right)^k b$ common neighbours.
\end{lemma}
\begin{proof}
Let $x=\left( \frac{\alpha}{e} \right)^k N$.  Suppose, for contradiction sake, that every $k$ vertices in $A$ have less than $x$ common neighbors in $[N]$. Then, by double-counting and convexity, note that \begin{equation*}\label{EquationGoesBrrrr}
x\binom{t}{k} > \sum_{v\in B}\binom{d(v)}{k} \geq  N \binom{\alpha t}{k}
\end{equation*}
Note that ${t \choose k}/{\alpha t \choose k} \leq \left(\frac{e t}{k} \right)^k / \left(\frac{\alpha t}{k} \right)^k$;
so, by definition of $x$, we have ${t \choose k}/{\alpha t \choose k} \leq \frac{N}{x}$, contradicting (\ref{EquationGoesBrrrr}). 
\end{proof}
\noindent The next statement is an easy corollary of the two above lemmas.
\begin{corollary}\label{cor:ramseyKST}
Let $\alpha \in (0,1)$ and $a,k, b \in \mathbb{N}$ be such that $a \geq 2k$. Then, every two-coloured $K_{a,b}$ with parts $A$ of size $a$ and $B$ of size $b$ contains a monochromatic copy of $K_{k, b/(2e)^k}$ with the part of size $k$ contained in $A$. Moreover, if $B$ is also $2$-coloured and $b \geq (4e)^k$, then there exist $A' \subseteq A,B' \subseteq B$ which form a monochromatic copy of $K_{k,k}$ and such that $B'$ is a monochromatic clique.
\end{corollary}
\begin{lemma}[Dependant Random Choice, \cite{drc}]\label{lem:drc}
Let $G$ be a bipartite graph with parts $A,B$ such that $e(A,B) \geq \delta|A||B|$. Then, if $a,m,r,s$ are such that
$$|B|\left(\frac{\delta}{e} \right)^{r} - |B|^s \cdot \frac{m^r}{|A|^r} \geq a ,$$
there exists a set $S \subseteq B$ of size $a$ such that every $s$ vertices in $S$ have at least $m$ common neighbours in $A$.
\end{lemma}

\begin{lemma}\label{lem:123}
Let $1/2 > \delta > 0$, $K_n$ be 2-coloured with red and blue and $A,B,C$ be subsets of size at least some $N \geq \delta^{-1000t}$. Suppose further that $e_r(A,B) \geq \delta|A||B|$ and for all $u \in B$ we have $d_b(u,C) \geq \delta |C|$. Then, there exist subsets $A' \subseteq A, B' \subseteq B, C' \subseteq C$ such that the following hold.
\begin{itemize}
    \item $B'$ is a monochromatic clique of size $10t$ and $|A'|,|C'| \geq N^{1/3}$.
    \item All the edges in $E[A',B']$ are red and all the edges in $E[B',C']$ are blue.
\end{itemize}
Moreover, the complete bipartite graph formed by $A',C'$ contains monochromatic copies of $K_{10t,N^{1/4}}$ with the part of size $10t$ inducing a monochromatic clique and contained in either side.
\end{lemma}
\begin{proof}
Note first that since $e_r(A,B) \geq \delta|A||B|$, we can directly apply Lemma \ref{lem:drc} in order to find a set $B_1 \subseteq B$ of size $N^{0.9}$ such that every $10t$ vertices in $B_1$ have at least $\sqrt{N}$ red common neighbours in $A$. Indeed, we can do this by setting $s := 10t, m :=\sqrt{N}, a := N^{0.9}$ and $r$ such that $|A|^{(r-1)/2} \leq |B|^{s} \leq |A|^{r/2}$ - so that in particular, $|B| \geq |A|^{(r-1)/2s} \geq N^{(r-1)/20t} \geq \delta^{-40r}$ and thus, $|B| \left(\frac{\delta}{e} \right)^r \geq |B|^{0.9} \geq N^{0.9}$. Secondly, since every vertex $u \in B_1 \subseteq B$ has $d_b(u,C) \geq \delta |C|$, it must be that $e_b(B_1,C) \geq \delta|B_1||C|$ and thus, we can again apply Lemma \ref{lem:drc} to find a set $B_2 \subseteq B_1$ of size at least $2^{20t}$ such that every $10t$ vertices in $B_1$ have at least $N^{1/3}$ blue common neighbours in $C$. We now apply Lemma \ref{lem:ramsey} to find a monochromatic clique $B' \subseteq B_2$ of size $10t$, which by construction must be such that its vertices have at least $\sqrt{N}$ red common neighbours in $A$ (which we take to be the set $A'$) and at least $N^{1/3}$ blue common neighbours in $C$ (which we take to be the set $C'$). Now, in order to find the desired monochromatic copies of $K_{10t,N^{1/4}}$ in the bipartite graph formed by $A',C'$, we can directly apply Corollary~\ref{cor:ramseyKST}.
\end{proof}
\begin{lemma}\label{lem:unav}
Let $K_n$ be $2$-coloured with red and blue and $A,B$ be two disjoint sets with the property that there is no red $K_{t}$ contained either $A$ or $B$. Moreover, suppose that the subgraph of red edges does not induce a $K_{t,t}$. Then, provided $|A|,|B|\geq \delta^{-C't}$ for some large enough constant $C'$, we have $e_r(A,B) \leq \delta^{20}|A||B|$.
\end{lemma}
\begin{proof}
Otherwise, we can apply Lemma \ref{lem:123} to find a monochromatic clique $A' \subseteq A$ of size $t$ with at least $2^{2t}$ red common neighbours in $B$. By assumption $A'$ is a blue clique and by Lemma \ref{lem:ramsey} its set of red common neighbours in $B$ contains a monochromatic clique of size $t$, which must also be blue. Observe that this creates an induced copy of $K_{t,t}$ in red, a contradiction. 
\end{proof}

\section{Proof of Theorem \ref{thm:1}}

\begin{proof}
Let $C$ be an arbitrarily large constant and $n \geq \eps^{-Ct}$. Suppose that $K_n$ is edge-coloured with red and blue so that $d_r(v),d_b(v) \geq n/4+\eps n$ for every vertex $v$. For the sake of contradiction, we will assume there is no induced copy of $K_{t,t}$ in either of the colours. Now, let $S^{r}_1, \ldots, S^r_{m_r}, S^{b}_1, \ldots, S^b_{m_b}$ be a collection of disjoint sets of size $\eps^{-C't}$, where $C'$ is the constant in Lemma \ref{lem:unav} applied with $\delta:= \eps$, such that each $S^r_i$ contains no red $K_{t}$ and each $S^b_i$ contains no blue $K_{t}$. Let $m := m_r m_b$ be maximal with respect to this.
\begin{lemma}
$m > 0$.
\end{lemma}
\begin{proof}
In order to prove this, note that we need only to show that there exist two disjoint sets $S^r,S^b$ of size at least $\eps^{-C't}$ such that $S^r$ contains no red $K_{t}$ and $S^b$ contains no blue $K_{t}$. First, by Lemma \ref{lem:ramsey}, there exists a monochromatic clique $K$ of size $10t$. Suppose that it is red (the other case is similar). Now, since every vertex has at least $n/4$ blue neighbours, by Lemma \ref{prop:KST} there is a subset $K' \subseteq K$ of size $t$ with at least $\left(\frac{1}{4e}\right)^t n \geq \eps^{-C't}$ blue common neighbours - we let $S^r$ be the set of these. Note indeed that $S^r$ cannot contain a red $K_t$ since otherwise, we would have an induced copy of a $K_{t,t}$ in blue. Similarly, since every vertex has at least $n/4$ red neighbours, by Lemma \ref{lem:123} there is a monochromatic clique $K' \subseteq S^r$ of size $t$, which is then blue, with at least $|S^r|^{1/3} \geq n^{1/4} \geq \eps^{-C't}$ red common neighbours - we let $S^b$ be the set of these. Note indeed that $S^b$ cannot contain a blue $K_t$ since otherwise, we would have an induced copy of a $K_{t,t}$ in red. Finally, observe that $S^r,S^b$ are disjoint by construction.
\end{proof}
\noindent We can now show the following.
\begin{lemma}\label{lem}
$m_r + m_b \geq (1-\eps^2)n$.
\end{lemma}
\begin{proof}
Suppose otherwise and take the corresponding sets. Let us use $V^r$ and $V^b$ to respectively denote $\bigcup_i S^{r}_i$ and $\bigcup_j S^{b}_j$. By the above lemma both of these are non-empty sets. Let also $Z$ denote the set of vertices not in any of these sets - by assumption, we have $|Z| \geq \eps^2 n$. First, note that by Lemma \ref{lem:unav} the following must hold.
\begin{itemize}
    \item For all $1 \leq i < j \leq m_r$, we have $e_r(S^{r}_i,S^r_j) < \eps^{20}|S^{r}_i||S^{r}_i|$. Similarly, for all $1 \leq i < j \leq m_b$, we have $e_b(S^{b}_i,S^b_j) < \eps^{20}|S^{b}_i||S^{b}_i|$.
    \item Consequently, $e_r(X^r) \leq \eps^{10}|X^r|n$ and $e_b(X^b) \leq \eps^{10}|X^b|n$.
\end{itemize}
This implies the following.
\begin{claim*}
There exists an $S^r_i$ such that $e_r(S^r_i,Z) \geq \eps|S^{r}_i|n/2$ or an $S^b_i$ such that $e_b(S^b_i,Z) \geq \eps|S^{b}_i|n/2$.
\end{claim*}
\begin{proof}
First note that we must have $e_r \left(X^r,X^b \right) \leq n|X^r|/4$ or $e_b \left(X^r,X^b \right) \leq n|X^b|/4$ since otherwise we would have $(|X^r|+|X^b|)^2/4 \geq |X^r||X^b| = e(X^r,X^b) > n|X^r|/4 + n|X^b|/4$. This in turn implies that $|X^r|+|X^b| > n$ which is a contradiction since the sets $X^r,X^b$ are disjoint. Suppose then w.l.o.g that $e_r \left(X^r,X^b \right) \leq n|X^r|/4$. Then, since every vertex has at least $n/4+ \eps n$ red neighbours, we must have that
$$(n/4+ \eps n)|X^r| \leq 2e_r(X^r) + e_r (X^r,X^b) + e_r (X^r,Z) \leq |X^r|(n/4+2\eps^{10}n) + e_r (X^r,Z) ,$$
and so, $e_r (X^r,Z) \geq \eps |X_r|n/2$ which by pigeonholing over the sets $S^r_i$ gives the desired outcome.
\end{proof}
\noindent We now show that the above claim is a contradiction to the maximality of $m$. Indeed, suppose that for $e_r(S^r_i,Z) \geq \eps|S^r_i|n/2 \geq \eps|S^{r}_i||Z|/2$. Then, since $|Z| \geq \eps^2n$, by Corollary \ref{cor:ramseyKST} there exists a blue clique $K \subseteq S^r_i$ of size $t$ whose vertices have at least $\eps^{-C't}$ red common neighbours in $Z$ - let $Z'$ denote the set of these. Then, note that $Z'$ cannot contain a blue clique of size $t$ since otherwise there would be an induced copy of $K_{t,t}$ in red. Hence, we can define $S^{b}_{m_b+1} := Z'$ and contradict the maximality of $m$.  
\end{proof}
To finish, we must have that $m_r+m_b \geq (1-\eps^2)n$ by the above lemma, and so, $|Z| \leq \eps^2n$. Further, observe (as in the beginning of the proof of Lemma \ref{lem}), that $e_r(X^r) \leq \eps^{10}|X^r|n$ and $e_b(X^b) \leq \eps^{10}|X^b|n$. But then, since every vertex has at least $n/4+\eps n$ red neighbours, we have that 
$$e_r(X^r,X^b) \geq (n/4+ \eps n)|X^r| - 2e_r(X^r) - e_r(X^r,Z) \geq (n/4+ \eps n)|X^r| - 2\eps^{10}|X^r|n - \eps^{2}|X^r|n > n|X^r|/4 .$$
Analogously, it also holds that $e_b(X^r,X^b) > n|X^r|/4$, which is a contradiction since then, $|X^r||X^b| > n(|X^r|+|X^b|)/4 \geq (|X^r|+|X^b|)^2/4$. Concluding, we achieve a contradiction, and thus, $K_n$ must contain an induced copy of $K_{t,t}$ in some colour.
\end{proof}
\section{Proof of Theorem \ref{thm:2}}
\begin{proof}
\noindent Let $n \geq \eps^{-Ct}$, where $C$ is an arbitrarily large constant. Suppose that $K_n$ is two-coloured with red and blue and that every vertex $v$ is such that $d_r(v),d_b(v) \geq \eps n$. Suppose for sake of contradiction that there is no $t$-locally-unavoidable pattern. The following is an observation already present in \cite{KM}.
\begin{observation}\label{obs:altC4}
There is no alternating homogeneous $t$-blow-up of $C_4$. 
\end{observation}
\noindent Now, let us take a random partition of $K_n$ into two sets $X,Y$. By Lemma \ref{lem:chernoff}, we can assume that $|X|,|Y| > n/3$ and every vertex has at least $\eps n/3$ neighbours in each colour and each part. The following is the crucial step of the proof. Define $I$ to be the minimal integer $i \geq 0$ such that the following condition is not satisfied.
\begin{itemize}    \item[(*)] There exists sets $X_1 \subseteq X$, $Y_1 \subseteq Y$ such that $|X \setminus X_1||Y \setminus Y_1| \leq 0.9^{i}n^2$ and sets $X_2 \subseteq X$, $Y_2 \subseteq Y$ of size at least $0.5^{i} \cdot \eps^{-100t}$ such that: every vertex $v \in X_2$ has $d_r(v,Y_1) \geq (1-10^{i}\eps^{30})|Y_1|$; and every vertex $v \in Y_2$ has $d_b(v,X_1) \geq (1-10^{i}\eps^{30})|X_1|$.
\end{itemize}
\noindent Let us briefly note why in particular, such an $I$ exists. First, it is clear that $i = 0$ satisfies condition (*) by taking $X_1,Y_1 = \emptyset$. Further, note the following.
\begin{claim}
$i = \lceil 4\log_{0.9} \left(\eps \right) \rceil$ does not satisfy condition (*). In particular, $I \leq \lceil 4\log_{0.9} \left(\eps \right) \rceil$.
\end{claim}
\begin{proof}
Suppose for sake of contradiciton that such $i := \lceil 4\log_{0.9} \left(\eps \right) \rceil$ satisfies (*). Then, we have that $|X \setminus X_1||Y \setminus Y_1| \leq \eps^4n^2$ and so, without loss of generality, $|X \setminus X_1| \leq \eps^2n$. In turn, $Y_2$ is non-empty and thus, we achieve a contradiction since by definition there is a vertex $v \in Y_2$ with at most $|X \setminus X_1| + 10^i \eps^{30}|X_1| \leq \eps^2 n + 10^i \eps^{30}n < \eps n/3$ red neighbours in $X$. This contradicts that every vertex has at least $\eps n/3$ red neighbours in $X$.
\end{proof}
\noindent To finish, we claim that the above claim leads to a contradiction. Indeed, take $I-1$, which by definition satisfies condition (*) and take the corresponding sets $X_1,X_2,Y_1,Y_2$. Since $i = I$ does not satisfy (*), we must have that $|X \setminus X_1||Y \setminus Y_1| > 0.9^{I}n^2$ and thus, $|X \setminus X_1|,|Y \setminus Y_1| > 0.9^{I}n \geq \eps^4 n$. Now, suppose w.l.o.g that red is the densest colour in $E[X \setminus X_1,Y \setminus Y_1]$, so that $e_r(X \setminus X_1,Y \setminus Y_1) \geq \frac{1}{2}|X \setminus X_1||Y \setminus Y_1|$.
\begin{lemma}\label{lem1}
There exists $U \subseteq Y \setminus Y_1$ and $V \subseteq X$ such that the following hold: $|U| \geq 0.1|Y \setminus Y_1|$, $|V|  \geq 2n^{1/10}$ and all $v \in V$ have $d_r(v,U) \geq (1-\eps^{50})|U|$.
\end{lemma}
\begin{proof}
First, observe that since $e_r(X \setminus X_1,Y \setminus Y_1) \geq \frac{1}{2}|X \setminus X_1||Y \setminus Y_1|$, there is a subset $X' \subseteq X \setminus X_1$ of size at least $\frac{1}{4}|X \setminus X_1| \geq \eps^4n/4$ such that all $v \in X'$ have at least $\frac{1}{4}|Y \setminus Y_1|$ red neighbours in $Y \setminus Y_1$. Now, since every vertex in $X' \subseteq X$ has at least $\eps n/3$ blue neighbours in $Y$ and every vertex in $Y$ has at least $\eps n/3$ red neighbours in $X$, we can directly apply Lemma \ref{lem:123} with $\delta = \eps/3$ and $A := X', B := Y, C = X$ in order to find sets $A' \subseteq X', B' \subseteq Y, C' \subseteq X$ such that $|A'|=|B'| = 10t$ are monochromatic cliques, $|C'| \geq n^{1/5}$ and $(A',B')$ is a complete bipartite graph in blue, $(B',C')$ is a complete bipartite graph in red and $(A',C')$ is a monochromatic complete bipartite graph. 

\noindent Now, since $A' \subseteq X'$, every vertex in $A'$ has at least $\frac{1}{4}|Y \setminus Y_1|$ red neighbours in $Y \setminus Y_1$ and so, $e_r(A',Y \setminus Y_1) \geq \frac{1}{8}|A'||Y \setminus Y_1|$. Therefore, there exists a subset $U \subseteq Y \setminus Y_1$ of size at least $\frac{1}{8}|Y \setminus Y_1|$ such that all $u \in U$ have $d_r(u,A') \geq \frac{1}{8}|A'| > t$. 
\begin{claim*}
$e_b(C',U) \leq \eps^{50}|U||C'|/2$.
\end{claim*}
\begin{proof}
Suppose otherwise. Recall that $|C'|,|U| \geq n^{1/5}$ and so, by Lemma \ref{lem:123} there exists a monochromatic clique $C'' \subseteq C'$ of size $t$ with at least $n^{1/30}$ blue common neighbours in $U$ - let this set of neighbours be denoted as $U'$. Recalling that every vertex in $U'$ has at least $t$ red neighbours in $A'$, by pigeonholing, there is a subset $U'' \subseteq U'$ of size at least $n^{1/30} / \binom{|A'|}{t} \geq n^{1/30} / \binom{10t}{t} \geq 20^{t}$ and $A'' \subseteq A'$ of size $t$ such that $U'', A''$ forms a red complete bipartite graph. Finally, by Corollary \ref{cor:ramseyKST}, since $B'$ is a monochromatic clique of size $10t$, there exists a monochromatic $K_{t,t}$ in the bipartite graph formed by $U'',B'$ whose parts are monochromatic cliques  - let it have parts $U''',B''$. We are now done since $A'',B'',C'',U'''$ form an alternating homogeneous $t$-blow-up of $C_4$, which contradicts Observation \ref{obs:altC4}.
\end{proof}
\noindent From the above claim we are done by taking $V$ to be the set of all vertices $v \in C'$ such that $d_b(v,U) \leq \eps^{50}|U|$ - this has size at least $|C'|/2 \geq n^{1/5}/2 \geq 2n^{1/10}$.
\end{proof}
\noindent We will now use the sets $U,V$ found in Lemma \ref{lem1} to show that $i = I$ must satisfy (*), which is a contradiction. This will be done by \emph{essentially} just re-setting $Y_1 = Y_1 \cup U$. Indeed, note first that the following must hold.
\begin{lemma}
$e_b(Y_1,V) \leq 4 \cdot 10^{I-1}\eps^{30}|Y_1||V|$ or $e_b(U,X_2) \leq 4 \cdot 10^{I-1}\eps^{30}|U||X_2|$.
\end{lemma}
\begin{proof}
Suppose for sake of contradiction that $e_b(Y_1,V) > 4 \cdot 10^{I-1}\eps^{30}|Y_1||V|$ and $e_b(U,X_2) > 4 \cdot 10^{I-1}\eps^{30}|U||X_2|$. First, let us note that there are at least $2 \cdot 10^{I-1}\eps^{30}|X_2|$ vertices in $X_2$ with at least $2 \cdot 10^{I-1}\eps^{30}|U|$ blue neighbours in $U$ - let $X'_2$ denote the set of these vertices. Observe also that since by assumption all vertices $v \in X'_2 \subseteq X_2$ have $d_r(v,Y_1) \geq (1-10^{I-1}\eps^{10})|Y_1|$, we have that $e_r(X'_2,Y_1) \geq (1-10^{I-1}\eps^{10})|Y_1||X'_2|$.

\noindent Secondly, there must also exist at least $2 \cdot 10^{I-1}\eps^{30}|Y_1|$ vertices in $Y_1$ with at least $2 \cdot 10^{I-1}\eps^{30}|V|$ blue neighbours in $V$ - let $Y_1' \subseteq Y_1$ denote the set of these vertices. Since by before $e_r(Y_1,X'_2) \geq (1-10^{I-1}\eps^{10})|Y_1||X'_2|$, we must have that $$e_r(Y_1',X'_2) \geq (1-10^{I-1}\eps^{10})|Y_1||X'_2| - |Y_1 \setminus Y_1'||X'_2| \geq  10^{I-1}\eps^{30}|Y_1||X'_2| \geq \eps^{30}|Y'_1||X'_2| .$$ 
Since $|X'_2|,|Y'_1|,|V| \geq \eps^{30} \cdot 0.5^{I-1} n^{1/10} \geq n^{1/20}$ (since $I \leq 10 \log_{0.9}(\eps)$), we can now apply Lemma \ref{lem:123} with $A := X'_2, B:= Y'_1, C := V$ to find sets $A' \subseteq X'_2, B' \subseteq Y'_1, C' \subseteq V$ such that $B', C'$ are monochromatic cliques of size $10t$, $|A'| \geq n^{1/100}$ and $(A',B')$ form a complete red bipartite graph, $(B',C')$ form a complete blue bipartite graph and $(A',C')$ form a monochromatic complete bipartite graph. 

\noindent Now, consider $C' \subseteq V$ and recall that by Lemma \ref{lem1} we must have $e_r(C',U) \geq (1-\eps^{50})|C'||U|$. This implies that there are at least $(1-2\eps^{50})|U|$ vertices in $U$ with at least $|C'|/2 \geq t$ red neighbours in $C'$ - let us denote this set as $U' \subseteq U$. Since by definition every vertex in $A' \subseteq X'_2$ has at least $2 \cdot 10^{I-1}\eps^{30}|U|$ blue neighbours in $U$ we have that $e_b(A',U') \geq 2 \cdot 10^{I-1}\eps^{30}|U||A'| - |U \setminus U'||A'| \geq \eps^{30}|U'||A'|$. To finish, since $|U'|,|A'| \geq n^{1/100}$ note that by Lemma \ref{lem:123} and pigeonholing, there exist a monochromatic clique $A'' \subseteq A'$ of size $t$ whose vertices have at least $n^{1/500}/ \binom{|C'|}{|C'|/2} \geq 20^t $ blue common neighbours in $U'$, all of which are red adjacent to the same subset $C'' \subseteq C'$ of size $t$ - let $U''$ denote the set of common neighbours. Finally, since $|B'| \geq 10t$, we can apply Corollary \ref{cor:ramseyKST} to find $B'' \subseteq B'$ and $U'' \subseteq U'$ which are monochromatic cliques of size $t$ and such that $(B',U'')$ forms a complete monochromatic bipartite graph. We are done since $A'',B'',C'',U''$ form an alternating homogeneous $t$-blow-up of $C_4$.
\end{proof}
\noindent To finish the proof of Theorem \ref{thm:2}, let us suppose w.l.o.g. that the first case of the previous lemma occurs, that is, that $e_b(U,X_2) \leq 4 \cdot 10^{I-1}\eps^{30}|U||X_2|$ - the other case is analogous. Note that this implies that the subset $X'_2 := \{v \in X_2: d_r(v,U) \geq (1-10^{I}\eps^{30})|U|\} \subseteq X_2$ has size at least $|X_2|/2$. We claim now that $Y'_1 := Y_1 \cup U, Y_2, X_1, X'_2 $ ensure that the previous conditions are satisfied with $i = I$, a contradiction. Indeed, note first that since $|U| \geq 0.1|Y \setminus Y_1|$, we have $|X \setminus X_1||Y \setminus Y'_1| \leq 0.9|X \setminus X_1||Y \setminus Y_1| \leq 0.9^{I}n^2$. Secondly, recall that $|X'_2| \geq |X_2|/2 \geq 0.5^{I} n^{1/10}$. Finally, by definition we have that every vertex $v \in X'_2$ has
$$d_r(v,Y'_1) \geq d_r(v,U) + d_r(v,Y_1) \geq (1- 10^{I}\eps^{30})|U| + (1-10^{I-1}\eps^{30})|Y_1| \geq (1- 10^{I}\eps^{30})|U||Y_1| ,$$
as desired.
\end{proof}

\section{Concluding remarks}
We have shown that there is some large constant $C$ such that for every $1/2 > \varepsilon>0$, any $2$-coloured $K_n$ on $n\geq \varepsilon^{-Ct}$ vertices with $d_{r,b}(v)\geq (1/4+\varepsilon)n$ for all vertices $v$ or $d_{r,b}(v)\geq \varepsilon n$ for all $v$, contains an induced copy of $K_{t,t}$ in some colour or a $t$-locally-unavoidable pattern, respectively. The bounds on $n$ are essentially tight for fixed $\varepsilon>0$ and with a bit more work the constant $C$ can be considerably reduced. Another way of looking at the problem is to fix $n$ and ask what is the smallest minimum degree in both colours which force either of the structures. It would be interesting to know whether our methods could be extended to prove the following Tur\'an-type problem.
\begin{problem}
Is there a constant $C(t)$ such that any $2$-coloured $K_n$ on $n$ vertices and minimum degree at least $n/4+C(t)n^{2-1/t}$ in both colours contains an induced $K_{t,t}$ in one of the colours? Similarly, if the minimum degree in each colour is at least $C(t)n^{2-1/t}$, does it contains a $t$-locally-unavoidable pattern?
\end{problem}

\bibliographystyle{plain}

\end{document}